\documentclass[11pt,a4paper,reqno]{amsart}
\usepackage{amsmath}
\usepackage{amsthm}
\usepackage{amsfonts}
\usepackage{amssymb}

\newcommand\A{{\mathcal{A}}}
\newcommand\B{{\mathcal{B}}}
\renewcommand\S{{\mathcal{S}}}

\newcommand\N{{\mathbb{N}}}

\theoremstyle{plain}
  \newtheorem{theorem}{Theorem}

  \newtheorem{lemma}{Lemma}
  
  \newtheorem{question}{Question}

\theoremstyle{plain}

\theoremstyle{definition}
  \newtheorem{definition}{Definition}
\include{psfig}

\begin{document}

\title{On a question of S\'{a}rk\"{o}zy on gaps of product sequences}
\author{Javier Cilleruelo}
\address{Departamento de Matem\'{a}ticas\\
Universidad Aut\'{o}noma de Madrid\\
28049 Madrid, Spain} \email{franciscojavier.cilleruelo@uam.es}
\author{Th\'{a}i Ho\`{a}ng L\^{e}}
\address{Department of Mathematics, UCLA, Los Angeles, CA 90095, USA}
\email{leth@math.ucla.edu}
\begin{abstract}Motivated by a question of S\'{a}rk\"{o}zy, we study the
gaps in the product sequence $\B=\A \cdot \A=\{b_n=a_ia_j,\
a_i,a_j\in \A\}$ when $\A$ has  upper Banach density $\alpha>0$. We
prove that there are infinitely many gaps $b_{n+1}-b_n\ll
\alpha^{-3}$ and that for $t\ge2$ there are infinitely many $t$-gaps
$b_{n+t}-b_{n}\ll t^2\alpha^{-4}$. Furthermore we prove that these
estimates are best possible.

We also discuss a related question about  the cardinality of the
quotient set $\A/\A=\{a_i/a_j,\ a_i,a_j\in \A\}$ when
$\A\subset\{1,\dots , N\}$ and $|\A|=\alpha N$.
\end{abstract}
\thanks{This work was developed during the Doccourse in Additive Combinatorics held in the Centre de Recerca
Matem\`{a}tica from January to March 2008. Both authors are extremely grateful for its hospitality. We would like also to thanks Terence Tao for reading a preliminary version of this paper and giving helpful comments.}
 \maketitle
\section{Introduction}
Let $\A=\{a_1 < a_2 < \ldots \}$ be an infinite sequence of positive
integers. The lower and upper asymptotic densities of $\A$ are
defined by
$$ \underline{d}(\A)=\liminf_{N \rightarrow \infty} \frac{ |\A \cap
\{1,\ldots, N\} |}{N}\quad\text{ and }\quad
\overline{d}(\A)=\limsup_{N \rightarrow \infty} \frac{ |\A \cap
\{1,\ldots, N\} |}{N}.$$ The lower and upper  Banach density of $\A$
are defined by
$$d_{*}(\A)=\liminf_{|I|\rightarrow \infty} \frac{ |\A \cap I
|}{|I|}\quad \text{ and }\quad d^{*}(\A)=\limsup_{|I| \rightarrow
\infty}  \frac{ |\A \cap I |}{|I|}$$ where $I$ runs through all
intervals. Clearly $d_{*}(\A) \leq \underline{d}(\A) \leq
\overline{d}(\A) \leq d^{*}(\A)$.

S\'{a}rk\"{o}zy considered the set $$\B= \A \cdot \A=\{b_1 < b_2 < \ldots
\}$$  of all products $a_{i}a_{j}$ with $a_{i}, a_{j} \in \A$ and
asked the following question, stated as problem 22 in
\cite{sarkozy}.
\begin{question} \label{q1}
Is it true that for all $\alpha >0$ there is a number
$c=c(\alpha)>0$  such that if $\A\subset \N$ is an infinite sequence
with $\underline{d}(\A)> \alpha$, then $b_{n+1}-b_{n} \le c$ holds
for infinitely many $n$?
\end{question}
This question is not trivial, since for any $0< \alpha <1$ and
$\epsilon >0$ there is a sequence $\A$ such that $\underline{d}(\A)
> \alpha
>0$ but $\bar{d}(\B) < \epsilon$, thus the gaps of $\B$ are greater
than $\frac{1}{\epsilon}$ on average. See the construction in
\cite{berczi}.

B\'erczi \cite{berczi} answered S\'{a}rk\"{o}zy's question in the
affirmative by proving that we can take $c(\alpha) \ll \alpha^{-4}$.
S\'{a}ndor \cite{sandor} improved it to $c(\alpha)\ll \alpha^{-3}$ even
assuming the weaker hypothesis $\overline{d}(\A)> \alpha$ .

\smallskip

In this work we consider  S\'{a}rk\"{o}zy's question for the upper
Banach density, that is to find a constant $c^*(\alpha)$ such that
$b_{n+1}-b_{n} \leq c^*(\alpha)$ infinitely often whenever
$d^{*}(\A)
> \alpha$. In this setting we can find the best possible value for
$c^*(\alpha)$ up to a multiplicative constant.

\begin{theorem}\label{t1}
For every $0<\alpha< 1$ and every sequence $\A$ with $d^{*}(\A) >
\alpha$, we have $b_{n+1}-b_{n} \ll \alpha^{-3}$
 infinitely often.
\end{theorem}

\begin{theorem}\label{t2}
For every $0<\alpha< 1$, there exists a sequence $\A$ with
$d^{*}(\A)
> \alpha$ and such that $b_{n+1}-b_{n} \gg \alpha^{-3}$ for every $n$.
\end{theorem}

We observe that, since $d^*(\A)\ge \overline d(\A)$, Theorem
\ref{t1} is stronger than S\'{a}ndor's result.

We also extend this question and study the difference
$b_{n+t}-b_{n}$ for a fixed $t$, namely to find a constant
$c^*(\alpha, t)$ such that $b_{n+t}-b_{n} \leq c^*(\alpha, t)$
infinitely often. Theorems \ref{t1} and \ref{t2} above correspond to the case $t=1$. For
greater $t$ the answer is perhaps surprising, in that the exponent
of $\alpha$ involved in $c^*(\alpha,t)$ is $-4$, not $-3$ like in
the case $t=1$.

\begin{theorem}\label{t3}
For every $0< \alpha < 1$, every $t\geq 2$ and every sequence $\A$
with $d^{*}(\A)>  \alpha$, we have $b_{n+t}-b_{n} \ll
t^2\alpha^{-4}$  infinitely often.
\end{theorem}

\begin{theorem}\label{t4}
For every $0<\alpha< 1$ and every $t \geq 2$, there is a sequence
$\A$ such that $d^{*}(\A)> \alpha$ and $b_{n+t}-b_{n} \gg
t^2\alpha^{-4}$ for every $n$.
\end{theorem}

\textbf{Notation.} We will denote by $\lceil x\rceil$ the smallest
integer greater or equal to $x$, $\lfloor x \rfloor$ the greatest
integer small than or equal to $x$. For quantities $A,B$ we write $A
\ll B$, or $B \gg A$ if there is an absolute constant $c>0$ such that $A \leq cB$.

\section{Proof of the results}
In our proof we will frequently use the following simple observation:
\begin{lemma}\label{L}Let $K$ be a positive integers and $\alpha$ a real number with $0<\alpha<1$. Then, if $d^*(\A)>\alpha$, there exists infinitely many disjoint intervals $I$ of length $K$ such that $|\A\cap I|\ge
\alpha |I|$.
\end{lemma}
\begin{proof}
Suppose for a contradiction, there exists at most a finite number of
intervals $I$ of length $K$ with $|\A\cap I|\ge \alpha K$. Thus,
there exists $N$ such that if $I\cap [1,N]=\emptyset$ and $|I|=K$
then $|\A\cap I|< \alpha |I|$.

Any interval $J$ can be written as an union of disjoint consecutive
intervals $$J=J_0\cup J_1\cup \cdots \cup J_r\cup J_{r+1},$$ where
$J_0=J\cap [1,N]$, $\ |J_i|=K,\ i=1,\dots , r$ and $|J_{r+1}|\le
K$.

We observe that \begin{eqnarray*}\frac{|\A \cap J|}{|J|}&=&\frac{|\A
\cap J_0|+|\A \cap J_1|+\cdots +|\A \cap J_r|+|\A \cap
J_{r+1}|}{|J|}\\ &<& \frac{N}{|J|}+\frac{\alpha(|J_1|+\cdots
|J_r|)}{|J|}+\frac{K}{|J|}<\frac{N+K}{|J|}+\alpha.\end{eqnarray*}
Since $\lim_{|J|\to \infty}\frac{N+K}{|J|}=0$ we obtain that
$d^*(\A)=\limsup_{|J|\to \infty}\frac{|\A \cap J|}{|J|}\le \alpha$, a contradiction.

Finally, it is clear that if there exist infinitely many intervals
$I$ of length $K$ with $|\A \cap I|\ge \alpha |I|$, there exist
infinitely many of them which are disjoint.
\end{proof}

\begin{proof}[Proof of Theorem \ref{t1}]
Let $L=\lceil 2\alpha^{-1}\rceil$. Since $d^{*}(\A)>  \alpha$, lemma
above with $k=L^2$ implies that there are  infinitely many disjoint
intervals $I$ of length $L^2$ such that $|I\cap \A| \geq \alpha
L^2$.

\smallskip

We divide each interval $I$ into $L$ subintervals of equal length
$L$. For $i=1, \ldots,L$, let $A_{i}$ be the number of elements of
$\A$ in the $i$-th interval. We count the number of differences
$a-a'$ where $0<a'<a$ are in the same interval. On the one hand, it
is
\begin{eqnarray*}\sum_{1\leq i \leq L} {A_{i} \choose 2}&=&\frac{1}{2}
\sum_{1\leq i \leq L} (A_{i}^2-A_{i})  \geq \frac{1}{2} \left(
\frac{1}{L} \left(\sum_{1 \leq i \leq L}
A_{i}\right)^2 - \sum_{1\leq i \leq L} A_{i} \right) \\
&=& \frac{1}{2} \left(\frac{|\A \cap I|^2}{L}  - |\A \cap I|\right)
= \frac{|\A \cap I|}{2} \left(\frac{|\A \cap I|}{L}  - 1\right)\\
&\ge & \frac{|\A \cap I|}{2} \left(\alpha L  - 1\right) =\frac{|\A
\cap I|}{2} \left(\alpha \lceil 2\alpha^{-1}\rceil  -
1\right)\\
&\ge & \frac{|\A \cap I|}{2}\ge \frac{\alpha L^2}2\ge L.
\end{eqnarray*}
On the other hand, the number of their possible values is at most
$L-1$. Thus we can find 2 couples $(a,a'),(a'',a''')$ such that
$0<a-a'=a''-a'''<L$. Then \begin{eqnarray*}0<
|aa'''-a'a''|&=&|a(a''+a'-a)-a'a''|\\ &=&|(a-a')(a''-a)|\\ &\le &
(L-1)(L^2-1)=(L-1)^2(L+1)\\&=&(\lceil 2\alpha^{-1}\rceil-1)^2(\lceil
2\alpha^{-1}\rceil+1)\\ &\le &4\alpha^{-2}
(2\alpha^{-1}+2)\\&<&4\alpha^{-2}
(2\alpha^{-1}+2\alpha^{-1})=16\alpha^{-3}.\end{eqnarray*}

Thus, each interval $I$ provides two consecutive elements of
$\B=\A\cdot \A$ with $b_{n+1}-b_n<16\alpha^{-3}$. Since there are
infinitely many of such intervals and they are disjoint we conclude
that $b_{n+1}-b_n\le 16\alpha^{-3}$ infinitely often.
\end{proof}
\begin{proof}[Proof of Theorem \ref{t3}]
Let $L=\lceil4t\alpha^{-2}\rceil$. Again, since $d^*(\A)>\alpha$, we
can apply Lemma \ref{L} with $K=L$ to deduce that there exist
infinitely many intervals $I$ of length $L$ which contain at least
$\alpha L$ elements of $\A$.

For each interval $I$, the number of sums $a+a',\ a\le a',\ a, a'
\in I \cap \A$ is at least $(\alpha L)^2/2$ and they are all
contained in an interval of length $2L$.

\smallskip

Since $\frac{(\alpha L)^2}2=2L\left (\frac{\alpha^2L}4\right
)=2L\left (\frac{\alpha^2\lceil4t \alpha^{-2}\rceil}4\right )> 2Lt$,
the pigeon hole principle implies that some sum $s$ must be obtained
in at least $t+1$ different ways,
$$s=a_1+a_1'=\cdots =a_{t+1}+a_{t+1}',\qquad a_i,a_i'\in I\cap \A.$$

If $i\ne j$, since $a_j+a_j'=a_i+a_i'$, we have
\begin{eqnarray*}0<|a_ia_i'-a_ja_j'|=|a_ia_i'-a_j(a_i+a_i'-a_j)|=(a_i-a_j)(a_i'-a_j)|<L^2,\end{eqnarray*} so the $t+1$ products $a_ia_i'$ lie in an
interval of length $$L^2< (4t\alpha^{-2}+1)^2\le
(5t\alpha^{-2})^2\le 25t^2\alpha^{-4}.$$

As in the proof of theorem \ref{t1}, each interval $I$ provides two
consecutive elements  of $\B=\A\cdot \A$ such that $b_{n+1}-b_n\le
25t^2\alpha^{-4}$. As in the proof of theorem \ref{t1} we can
conclude that $b_{n+1}-b_n\le 25t^2\alpha^{-4}$ infinitely many
times.
\end{proof}

In the proofs of Theorems \ref{t2} and \ref{t4}, we will take $\A$ to be a union of blocks sufficiently far apart from one another, so that small differences $b_{i+1}-b_{i}$ (or $b_{i+t}-b_{i}$) can only arise when the $b_{i}$ in question are made up from elements in the same block. To make this precise let us make the following

\begin{definition} Given a positive value
$x_1$ and an infinite sequence of finite sets of non negative
integers $\A_1,\A_2,\dots $ we define the associated sequence $\A$
to these inputs by
\begin{eqnarray}\label{AS}\A=\bigcup_{n=1}^\infty(x_n+\A_n),\end{eqnarray}
where the sequence $(x_n)$ is defined for $n\ge 2$ by
\begin{equation}\label{xn}x_n=x_1+M_n^2+M_n(x_{n-1}+M_{n-1})+(x_{n-1}+M_{n-1})^2\end{equation} and
$M_n$ is the largest element of $\A_n$.
\end{definition}
Clearly all the sets $x_n+\A_n$ in (\ref{AS}) are disjoint. Let us now verify that small gaps in $\B$ can only come from products of elements in the same block $x_n+\A_n$.

\begin{lemma}\label{A} Let $\A$ be defined as in (\ref{AS}). Then, all the nonzero differences
$d=c_1c_2-c_3c_4,$ with $c_1,c_2,c_3,c_4\in \A$ but not all $c_i$ in
the same $x_n+\A_n$, satisfy $|d|\ge x_1$.
\end{lemma}
\begin{proof}Let $n$ be the largest integer such that  $c_i\in x_n+\A_n$ for some
$i=1,2,3,4$. We can assume that $c_1\in \A_n$. Then there are many possibilities for $c_2, c_3, c_4$. It is a routine to check that the inequality $|d|\ge x_1$ holds in all these cases. We will use repeatedly the definition of $x_n$
in (\ref{xn}) and the fact that if $c\in x_m+\A_m$ then $x_m\le c\le
x_m+M_m$.
\begin{itemize}
     \item [i)] $c_2\in x_n+\A_n$ and $c_3$ or $c_4\not \in x_n+\A_n$.
In this case \begin{eqnarray*}|d|&\ge & x_n^2-|c_3c_4|\\&\ge &
x_n^2-(x_n+M_n)(x_{n-1}+M_{n-1})\\
&= &x_n(x_n-x_{n-1}-M_{n-1})-M_n(x_{n-1}+M_{n-1})\\ &\ge &
x_n-M_n(x_{n-1}+M_{n-1})\ge x_1.
\end{eqnarray*}
    \item [ii)] $c_2,c_3,c_4\not \in x_n+\A_n$.
    In this case $$|d|\ge x_n-c_3c_4\ge x_n-(x_{n-1}+M_{n-1})^2\ge
    x_1.$$
     \item [iii)] $c_3 \in x_n+\A_n$ and $c_2,c_4\not \in x_n+\A_n$.

     In this case we write $c_1=x_n+a_1$ and $c_3=x_n+a_3$. Then
     $$|d|=|x_n(c_2-c_4)+a_1c_2-a_3c_4|.$$

If $c_2=c_4$, then $|d|=c_2|a_1-a_3|\ge x_1.$

If $c_2\ne c_4$, then
$$|d|\ge x_n-|a_1c_2-a_3c_4|\ge
     x_n-M_{n}(x_{n-1}+M_{n-1})\ge x_1,$$
     since $|a_1c_2-a_3c_4|\le \max\{a_1c_2,a_3c_4\}\le
     M_n(x_{n-1}+M_{n-1})$.
\end{itemize}
\end{proof}
In order to prove Theorems \ref{t2} and \ref{t4}, we also need the
following construction of Sidon sets due to Erd\H{o}s and Tur\'{a}n
\cite{et}:
\begin{lemma}\label{ET}Let $p$ be an odd prime number. Let $$\S=\{s_{i}=2pi+(i^2)_{p}:i=0, \ldots,
p-1 \},$$ where $(x)_{p} \in [0, p-1]$ is the residue of $x$ modulo
$p$. Then $\S$ is a Sidon set in $[0,2p^2)$ with $p$ elements and
$|s_i-s_j| \geq p$ for every $i\ne j$.
\end{lemma}
\begin{proof}
It is clear that $$|s_i-s_j| \geq 2p|i-j|-|(i^2)_p-(j^2)_p|\ge p.$$
Suppose we have an equation $s_{i}+s_{j}=s_{k}+s_{l}$ for some
$i,j,k,l$.  Then
$$2p(i+j-k-l)=(i^2)_{p}+(j^2)_{p}-(k^2)_{p}-(l^2)_{p}.$$ The left hand
side is a multiple of $2p$ while the right hand side is strictly
smaller than $2p$. Thus $$i+j-j-l = 0$$ and $$
(i^2)_{p}+(j^2)_{p}-(k^2)_{p}-(l^2)_{p}=0,$$ i.e., $$i^2+j^2 \equiv
j^2+l^2 \pmod{p}.$$ Thus $$i^2+j^2 - k^2+l^2 = (i-k)(i+k-j-l) \equiv
0 \pmod{p}.$$ Either $i=k$ and $j=l$, or $i+k-j-l \equiv 0
\pmod{p}$, in which case $k=l$ and $i=j$.
\end{proof}
\begin{proof}[Proof of Theorem \ref{t2}]
For $\alpha\ge 1/16$ we take
$\A=\N$. Obviously $d^*(\A)=1\ge \alpha$ and all the gaps in
$\A\cdot \A$ are $\ge 1\ge 2^{-12}\alpha^{-3}$.

For $\alpha <1/16$, let $p$ be an odd prime such that $ \dfrac
1{8\alpha}<p<\dfrac 1{4\alpha}$, $\S$ the Sidon set defined in Lemma
\ref{ET} and $m=2p^2$. We consider the sequence $\A$ defined in
(\ref{AS}) with $x_1=4p^3$ and
\begin{equation} \A_n=\bigcup_{k=1}^n(2km+\S).
\end{equation}

First we observe that $\A_n$ is contained in the interval $I_n=[2m,
2mn+m)$ and then
$$d^*(\A)\ge \limsup_{n\to \infty}\frac{|\A_n|}{|I_n|}= \limsup_{n\to
\infty}\frac{|np|}{|(2m-1)n|}> \frac 1{4p}\ge \alpha.$$

Next we will prove that all the nonzero differences
$d=c_1c_2-c_3c_4$ with $c_1,c_2,c_3,c_4\in \A$ satisfy   $|d|\ge
4p^3$, and clearly $|d|\ge 2^{-7}\alpha^{-3}$.

By Lemma \ref{A} it is true when not all $c_i$ belong to the same
$x_n+\A_n$. Suppose then that $c_i=x_n+a_i,\ i=1,2,3,4$. Then
\begin{eqnarray*}d&=&(x_n+a_1)(x_n+a_2)-(x_n+a_3)(x_n+a_4)\\&=&x_n(a_1+a_2-a_3-a_4)+a_1a_2-a_3a_4.\end{eqnarray*}

\begin{itemize}
    \item If $a_1+a_2\ne a_3+a_4$ then $$|d|\ge x_n-|a_1a_2-a_3a_4|\ge
x_n-M_n^2\ge x_1=4p^3.$$
    \item If $a_1+a_2=a_3+a_4$ then
\begin{eqnarray*}|d|&=&|a_1a_2-a_3a_4|\\&=&|a_1a_2-a_3(a_1+a_2-a_3)|\\&=&|(a_2-a_3)(a_1-a_3)|.\end{eqnarray*}
Now we write $a_i=2k_im+s_i,\ 1\le k_i\le n,\ s_i\in \S.$ The
condition $a_1+a_2=a_3+a_4$ implies
$$2m(k_1+k_2-k_3-k_4)=s_3+s_4-s_1-s_2.$$ Since
$|s_1+s_2-s_3-s_4|<2m$, we have $k_1+k_2=k_3+k_4$ and
$s_1+s_2=s_3+s_4.$
 Now we use the fact that $\S$ is a Sidon set to
conclude that $\{s_1,s_2\}=\{s_3,s_4\}$. We can assume that
$s_1=s_3$ and $s_2=s_4$, Then
$$|d|=|2m(k_2-k_3)+(s_2-s_3)||2m(k_1-k_3)|.$$
\begin{itemize}
    \item If $s_2=s_3$, since $d\ne 0$ we have that $$|d|\ge
    (2m)^2\ge 16p^4>4p^3.$$
    \item If $s_2\ne s_3$, by Lemma \ref{ET} we know that $$p\le
    |s_2-s_3|<m.$$
    \begin{itemize}
        \item If $k_2\ne k_3$
then $|d|\ge |2m-m||2m|=2m^2=8p^4>4p^3$.
        \item If $k_2=k_3$ then $|d|\ge
p(2m)=4p^3$.
    \end{itemize}
    \end{itemize}
\end{itemize}

    In any case $|d|\ge 4p^3$.
\end{proof}

\begin{proof}[Proof of Theorem \ref{t4}]
For\footnote{The reason why we have to consider two cases $\alpha\ge 1/16$ and $\alpha< 1/16$ separately is that we require the exact inequality $d^{*}(\A)>\alpha$. If we are happy with, say, $d^{*}(\A) \gg \alpha$, then there is no need to consider 2 cases.} $\alpha\ge 1/16$ we consider the sequence $\A$ defined in
(\ref{AS}) with $x_1=t^2$ and $\A_n=\{1,\dots ,n\}$. Clearly
$d^*(\A)=1 > \alpha$.

Next, let $c_0c_0',\dots ,c_tc_t'$ be distinct elements in $\A\cdot
\A$. We will prove that
$$|c_ic_i'-c_jc_j'|\ge t^2/36$$ for some $i,j,\ i\ne j$.

In view of Lemma \ref{A}, we need only to consider the case where all the $c_i,c_i'$ belong to the same
$x_n+\A_n$.

 The inequality is obviously true for $2\le t\le 6$. Suppose $t\ge 7$. We
write
\begin{eqnarray*}d_i=c_0c_0'-c_ic_i'&=&(x_n+a_0)(x_n+a_0')-(x_n+a_i)(x_n+a_i')\\ &=&x_n(a_0+a_0'-a_i-a_i')+a_0a_0'-a_ia_i'.\end{eqnarray*}
If the coefficient of $x_n$ is non zero then $|d_i|\ge x_n-M_n^2\ge
x_1=t^2$.

We suppose then that $a_0+a_0'-a_i-a_i'=0$ for all $i=1,\dots ,t$.
It implies that $a_i\ne a_j$ if $i\ne j$ (since if not,
$c_ic_i'=c_jc_j'$). Then we have
\begin{eqnarray*}|c_0c_0'-c_ic_i'|&=&|a_0a_0'-a_ia_i'|\\
&=&|a_0a_0'-a_i(a_0+a_0'-a_i)|\\
&=&|(a_0'-a_i)(a_0-a_i)|.\end{eqnarray*}

Since there are at most $2(1+2(t/6))<t$ values of $i$ for which
$|a_0-a_i|\le t/6$ or $|a_0'-a_i|\le t/6$  we obtain
$$|a_0'-a_i||a_0-a_i|> (t/6)^2\ge 2^{-22}t^2\alpha^{-4}$$ for some
$i$.

For $0<\alpha <1/16$ we take the same sequence $\A$ used in the
proof of Theorem \ref{t2} but  with $x_1=t^2p^4.$ As we saw, this
sequence has density $d^*(\A)\ge \alpha$.  As in that proof, we
apply Lemma \ref{A} to see that if $c_i,c_i',c_j,c_j'$ not in the
same $x_n+\A_n$ for some $i\ne j$ then $|c_ic_i'-c_jc_j'|\ge
x_1=t^2p^4$ and we are done because $t^2p^4\ge
2^{-12}t^2\alpha^{-4}$.

Therefore, if $c_0c_0',\dots ,c_tc_t'$ are distinct elements of
$\A\cdot \A$, we can assume that all $c_i,c_i'$ belong to the same
$x_n+\A_n$ and we write them as $c_i=x_n+a_i,\ a_i\in \A_n$. Then
$$d_i=c_0c_0'-c_ic_i'=x_n(a_0+a_0'-a_i-a_i')+a_0a_0'-a_ia_i'$$
If $a_i+a_i'\ne a_0+a_0'$ for some $i\ne 0$ then $$|d_i|\ge
x_n-M_n^2\ge x_1=t^2p^4.$$ So we assume that $a_i+a_i'=a_0+a_0'$ for
all $i=0,\dots ,t$. We write $a_i=2mk_i+s_i$ and we can assume that
$s_i\le s_i'$ for $i=0,\dots , t$.
    The condition $a_i+a_i'=a_0+a_0'$ for
all $i=0,\dots ,t$ implies
    that $2m(k_i+k_i'-k_0-k_0')=s_0+s_0'-s_i-s_i'$ and since
    $|s_0+s_0'-s_i-s_i'|<2m$, we have $k_i+k_i'=k_0+k_0'$ and
    $s_i+s_i'=s_0+s_0'$.

    Since $S$ is a Sidon set and $s_i\le s_i'$ we have $s_i=s_0$
    and $s_i'=s_0'$ for $i=0,\dots , t$. Then
    $$c_ic_i'-c_0c_0'=2m(k_i-k_0)(2m(k_i-k_0')+s_0-s_0').$$
    We observe that all $k_i$ are distinct and $k_i\ne 0$.
(Otherwise, if $k_i=k_j$ then $k_i'=k_j'$ and then
$c_ic_i'=c_jc_j'$.)

            Suppose  $k_i\ne k_0'$. Then
            \begin{equation*}\label{p4}|c_ic_i'-c_0c_0'|=|2m(k_i-k_0)(2m(k_i-k_0')+s_0-s_0')|\end{equation*}
\begin{itemize}
    \item If $s_0=s_0'$ then
    \begin{eqnarray*}
|c_ic_i'-c_0c_0'|&=&4m^2|k_i-k_0||k_i-k_0'|\\ &\ge &
16p^4|k_i-k_0||k_i-k_0'|.
            \end{eqnarray*}
    \item If $s_0\ne s_0'$, since $|s_0-s_0'|\le m$ we have
    \begin{eqnarray*}
|c_ic_i'-c_0c_0'|&\ge &
            2m|k_i-k_0|(2m|k_i-k_0'|-m) \nonumber \\ &\ge &2m^2|k_i-k_0||k_i-k_0'|\nonumber\\&\ge &8p^4|k_i-k_0||k_i-k_0'|.
            \end{eqnarray*}
\end{itemize}
In both cases we have $$|c_ic_i'-c_0c_0'|\ge
8p^4|k_i-k_0||k_i-k_0'|.$$

If $2\le t\le 6$ we consider $k_1$ and $k_2$. One of them (or both)
is distinct from $k_0'$. For that $k_i$ we have
$|c_0c_0'-c_ic_i'|\ge 8p^4\ge 2^{-9}\alpha^{-4}\ge
2^{-14}t^2\alpha^{-4}$.

If $t\ge 7$ we observe that there are at most $2(1+2(t/6))<t$ values
of $i$ such that $|k_0-k_i|\le t/6$ or $|k_0'-k_i|\le t/6$. So there
exists some $i$ such that
$$|c_0c_0'-c_ic_i'|\ge 8p^4(t/6)(t/6)\ge 2^{-14}t^2\alpha^{-4}.$$
\end{proof}

\section{A related question}

We do not know if the exponent $-3$ in Theorem \ref{t1} can be
improved when $\overline{d}(\A)>\alpha$ or when
$\underline{d}(\A)>\alpha$, which is the original problem of
S\'{a}rk\"{o}zy. Clearly nothing better than $-2$ is possible. We
present an alternative approach to this question, which gives the
bound of G. B\'{e}rczi quickly.

Let $\A\subset \{1,\dots , N\}$ a set with $\alpha N$ elements. We
consider the set $$\A/\A=\{ a/a',\ a<a',\ a,a'\in A\}.$$ What can we
say about the cardinality of $\A/\A$ when $N$ is large? Clearly $|\A
/\A|\ll \alpha^2 N^2$. Probably it is the true order of magnitude
but we do not know how to improve the theorem below

\begin{theorem}\label{t5}If $\A\subset \{1,\dots , N\}$ with $|\A|=\alpha N$, then $|\A/ \A|\gg \alpha^4N^2$.
\end{theorem}
\begin{proof}
Let $(\A\times\A)_{d}= \{(a,a') \in \A\times\A: a<a', \gcd (a,a')=d
\}$. Then for every $d$, all the quotients $a/a',\ (a,a') \in
(\A\times\A)_{d}$ are distinct and contained in [0,1]. We first show
that there exists $d$ such that $|(\A\times\A)_{d}| \geq
\frac{\alpha^4}{9}N^2$. Let $T$ be an integer to be chosen later.
Then
\begin{eqnarray}
(\alpha N)^2 \leq |\A|^2 &=& \sum_{d} |(\A\times\A)_{d}| \nonumber \\
&=& \sum_{d \leq T} |(\A\times\A)_{d}| + \sum_{d>T} |(\A\times\A)_{d}| \nonumber \\
&\leq & T \max_{d \leq T} |(\A\times\A)_{d}| + \sum_{d>T} \left(\frac{N}{d}\right)^2 \nonumber \\
&\leq & T \max_{d \leq T} |(\A\times\A)_{d}| + \frac{N^2}T \nonumber
\end{eqnarray}
Thus there exists $d \leq T$ such that $$|(\A\times\A)_{d}| \geq N^2
 \left (\frac{\alpha^2}{T}-\frac{1}{T^2}\right ).$$ If we choose $T=\lceil
\frac{2}{\alpha^2}\rceil$ and observe that $T<\frac{3}{\alpha^2}$
when $\alpha<1$ we obtain $\frac{\alpha^2}{T}-\frac{1}{T^2} \geq
\frac{1}{T^2} \ge \frac{\alpha^4}{9}.$ Thus for some $d$,\
$|(\A\times\A)_{d}| \geq N^2 \alpha^4/9$.

Finally we observe that $|\A /\A|\ge |(\A\times\A)_{d}|$ for any
$d$.
\end{proof}
We observe that if $\overline{d}(\A)>\alpha$ there exist infinitely
many intervals $[1,N]$ such that $|\A\cap [1,N]|>\alpha$. Theorem
above and the pigeon hole principle implies that there are
$a/a',a''/a'''\in \A/\A$ such that $$\left |\frac
a{a'}-\frac{a''}{a'''}\right |\le 9\alpha^{-4}N^{-2},$$ so
$|aa'''-a'a''|\le 9\alpha^{-4}$.

Theorem \ref{t5} motivates the following questions for sets
$\A\subset \{1,\dots ,N\}$ with $|\A|=\alpha N$:
\begin{question} \label{q2} Is it true that for some $d$,\ $|(\A\times\A)_{d}|
\gg \alpha^2 N^2$?
\end{question}

\begin{question} \label{q3}
Is it true that $|\A/\A| \gg \alpha^2 N^2$ ?
\end{question}
Clearly an affirmative answer to Question \ref{q2} will answer
Question \ref{q3} which in turn answers Question \ref{q1}.

\end{document}